\documentclass[a4paper]{article}
\usepackage{amsmath, amsthm, amssymb, graphicx, booktabs}

\newcommand{\eps}{\varepsilon}
\newcommand{\prob}[1]{\mathbb{P}\left(#1\right)}
\newcommand{\mean}[1]{\mathbb{E}\left(#1\right)}
\newcommand{\bin}[1]{\operatorname{Bin}\left(#1\right)}
\newcommand{\deriv}[3][]{\frac{\mathrm{d}^{#1}#2}{\mathrm{d}#3^{#1}}}
\newcommand{\pderiv}[3][]{\frac{\partial^{#1}#2}{\partial#3^{#1}}}

\newcommand{\ee}{\mathrm{e}}
\newtheorem{thm}{Theorem}
\newtheorem{lem}[thm]{Lemma}

\title{Preferential attachment with choice}
\author{John Haslegrave and Jonathan Jordan\\
 University of Sheffield}

\begin{document}

\maketitle
\begin{abstract} We consider the degree distributions of preferential attachment random graph models with choice similar to those considered in recent work by Malyshkin and Paquette and Krapivsky and Redner.  In these models a new vertex chooses $r$ vertices according to a preferential rule and connects to the vertex in the selection with the $s$th highest degree.  For meek choice, where $s>1$, we show that both double exponential decay of the degree distribution and condensation-like behaviour are possible, and provide a criterion to distinguish between them.  For greedy choice, where $s=1$, we confirm that the degree distribution asymptotically follows a power law with logarithmic correction when $r=2$ and shows condensation-like behaviour when $r>2$.
\\ AMS 2010 Subject Classification: Primary 05C82. \\ Key words and phrases:random graphs; preferential
attachment; choice.\end{abstract}
\section{Introduction}

Since the introduction of the preferential attachment model for growing random graphs by Barab\'{a}si and Albert \cite{scalefree1999} and its mathematical analysis by Bollob\'{a}s, Riordan, Spencer and Tusn{\'a}dy \cite{BRST}, a number of variants on the model have been investigated.  One variant which has been of some recent interest, studied by Krapivsky and Redner \cite{KR2013} and Malyshkin and Paquette \cite{MalPaq1,MalPaq2}, is where a new vertex which joins the graph selects a number of potential neighbours at random according to the preferential attachment rule, and then chooses which one to connect to according to a deterministic criterion.

In \cite{MalPaq1} a new vertex joining the graph picks two vertices at random with probability proportional to their degrees (the preferential attachment rule) and then connects to the one with the smaller degree; this model is referred to as the \emph{min-choice preferential attachment tree}.  Similarly in \cite{MalPaq2} the analogous model where the new vertex picks the larger of the two vertices is considered, and this is the \emph{max-choice preferential attachment tree}.  In \cite{KR2013} this is generalised to the new vertex picking $r$ vertices at random according to the preferential attachment rule and then picking the one with the $s$th highest degree to connect to; it is this latter model we will consider.  Where $s=1$ this is referred to in \cite{KR2013} as \emph{greedy choice}, and where $s>1$ this is referred to in \cite{KR2013} as \emph{meek choice}.

The main interest in \cite{KR2013,MalPaq1} is in the asymptotic degree distributions of these graphs.  There are various possibilities.  For meek choice with $r=s=2$ (that is, min-choice preferential attachment) it is shown in \cite{MalPaq1} that the degree distribution has a doubly exponential decay, and it is conjectured in \cite{KR2013} that this applies for all cases of meek choice.  For greedy choice, \cite{KR2013} suggests that there is differing behaviour when $r=2$ from when $r>2$; in the former case there is a limiting degree distribution with a tail which is a power law with index $-2$ with a logarithmic correction, while in the latter case the limit is degenerate, in the sense that the limiting degree distribution sums to a value strictly less than $1$, and a single vertex (a \emph{macroscopic hub}) has a degree which grows as a positive fraction of the size of the graph.  This degenerate limit behaviour has some resemblance to the \emph{condensation} phenomenon observed for preferential attachment with fitness by Borgs, Chayes, Daskalakis and Roch \cite{fitness}, in that in the limit a proportion of vertices tending to zero take a positive proportion of the edges.

We will confirm the conjectures of \cite{KR2013} in the case of greedy choice and in some cases of meek choice.  However, for other cases of meek choice (in particular for $s=2$ and $r\geq 7$), we will show that in fact the behaviour is that of a degenerate limit with condensation similar to that for greedy choice with $r>2$, not the doubly exponential decay suggested by \cite{KR2013}.

The layout of the paper is as follows. In Section 2 we define the model and motivate the definition of the sequence $p_k$. In Section 3 we prove that this sequence is well-defined and describes the limiting distribution of a single preferential choice. In Section 4 we establish the limit of the $p_k$; we have a non-degenerate limiting distribution if and only if this limit is 1, and we bound the values of $r$ (in terms of $s$) for which this occurs. In Section 5 we prove that, for $s\geq 2$, the distribution will either be degenerate or have a doubly-exponential tail. In Section 6 we consider the case $r=2,s=1$, which is the only non-degenerate case for $s=1$ other than standard preferential attachment.

To some extent our methods follow, and can be seen as a generalisation of, those of \cite{MalPaq1}, where Malyshkin and Paquette analyse the case $r=s=2$. By contrast, \cite{MalPaq2} uses different methods and is largely concerned with the maximum degree for $s=1$. In Section 3 we prove that the probability of a single preferential choice having degree at most $k$ converges to $p_k$. We show this by induction; provided convergence occurs for $k$ then eventually the expected change in the probability for $k+1$ will tend to correct any difference from $p_{k+1}$. This convergence result, and the doubly-exponential decay in Lemmas 9 and 10, are similar to those obtained in \cite{MalPaq1}. However, while they are able to check by direct calculation that $1-p_{10}$ is small enough to start the doubly-exponential decay, we must prove that there is some suitable starting point in every non-degenerate case. We do this by showing that if there is decay to 0 then it is at least quadratic. Our methods in Section 4 are new; we prove that the $p_k$ converge to the smallest positive root of a certain polynomial function and show that this function is always positive if $r\leq 2s$, but that it becomes negative at a suitably-chosen point if $r$ is slightly larger than this, thus obtaining the dichotomy between doubly-exponential decay and condensation-like behaviour.

\section{Definitions and results}
Fix integers $r,s$ with $r\geq s\geq 1$. We grow a tree, starting from the two-vertex tree at time 1. At each time step we select an ordered $r$-tuple of vertices (with replacement, so that the same vertex may appear more than once), where each choice is independent and vertices are selected with probability proportional to their degree. We then add a new vertex attached to the vertex with rank $s$ (by degree) among the $r$ vertices chosen, breaking ties uniformly at random. We will generally think of $r$ being at least 2, since the case $r=s=1$ is standard preferential attachment.

\begin{figure}[ht]
\begin{center}
\begin{tabular}{ccc}
\includegraphics[width=.3\textwidth]{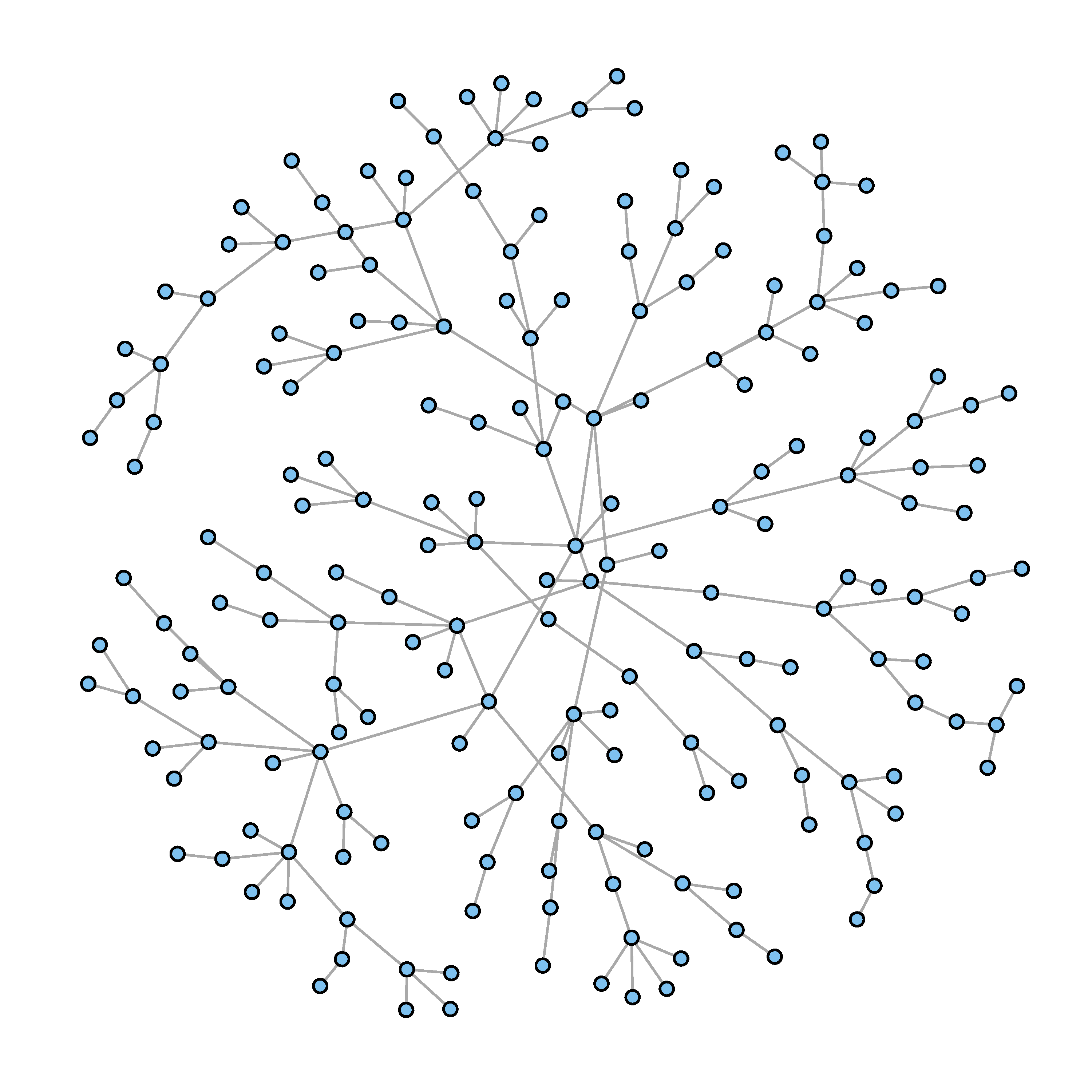} & \includegraphics[width=.3\textwidth]{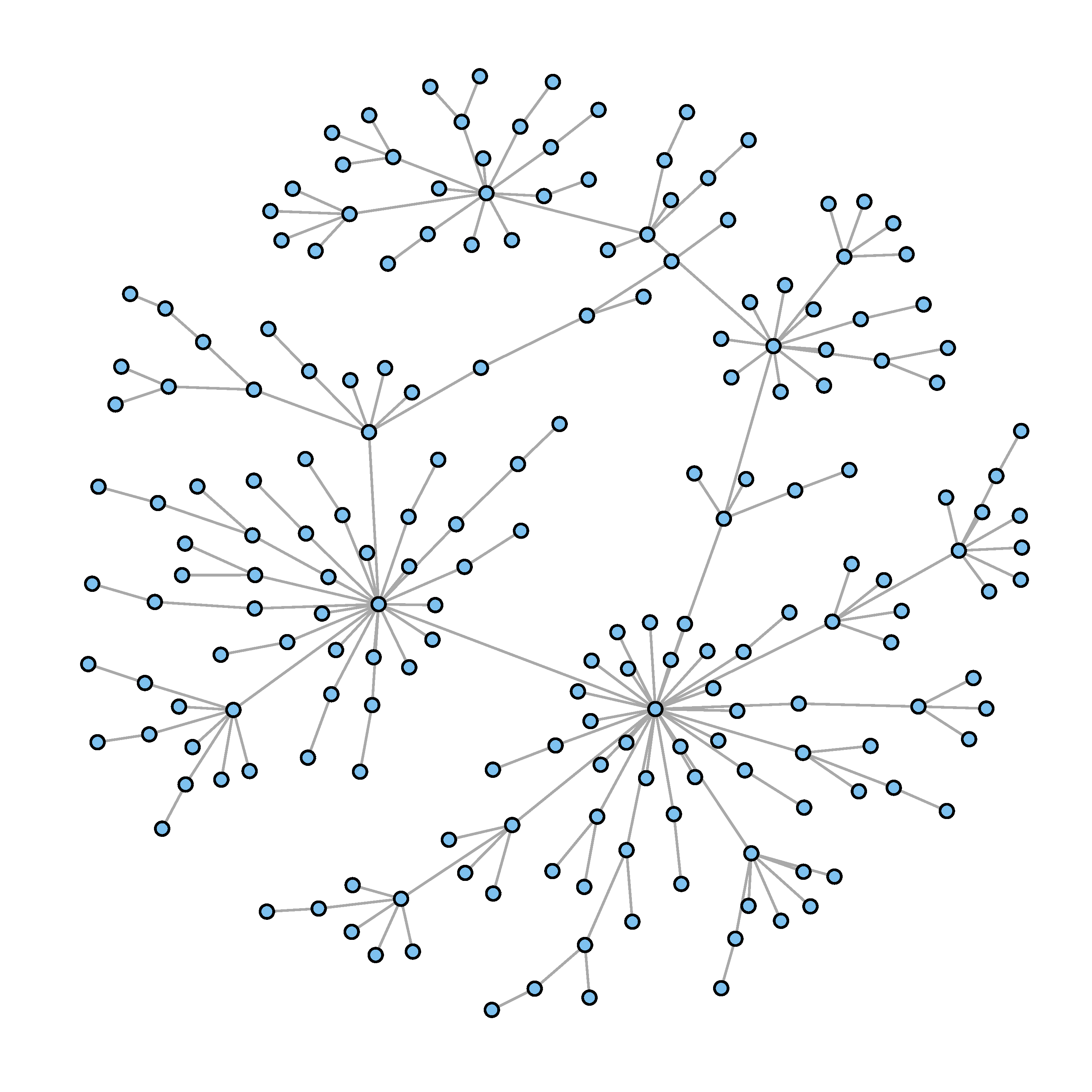} & \includegraphics[width=.3\textwidth]{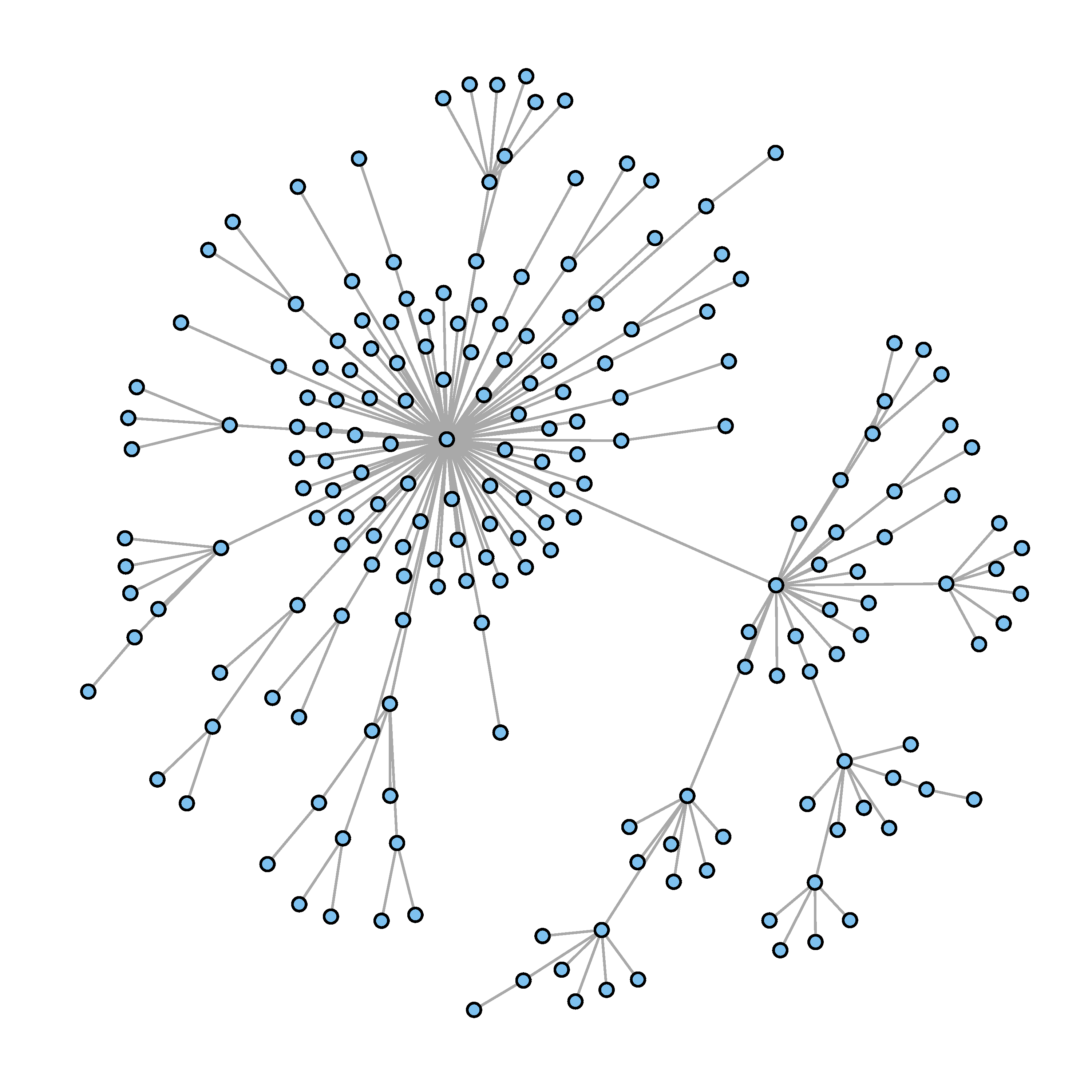}
\end{tabular}
\caption{200-vertex simulations, produced using igraph in R. Left to right: minimum choice, standard preferential attachment, maximum choice. The maximum degrees in these simulations are are 6, 30 and 90 respectively.}
\end{center}
\end{figure}

As the graph grows, the distribution of the degree of a single vertex obtained by a preferential choice will change. We show that it approaches a particular distribution, which depends only on $r$ and $s$, with high probability. We consider the limiting probability that a single preferential choice will give a vertex of degree at most $k$. 

Define $B_{r,s}(p)$ to be the probability that a $\bin{r,p}$ random variable takes a value greater than $r-s$. Write $F_m(k)$ for the sum of degrees of vertices with degree at most $k$ at time $m$, so that $F_m(k)/2m$ is the probability of selecting a vertex of degree at most $k$ with a single preferential choice. 

\begin{lem}\label{markov}The value of $F_m(k)$ evolves as a Markov process where
\[
F_{m+1}(k)-F_m(k)=
\begin{cases}
1 &\text{with prob }\quad 1-B_{r,s}(F_m(k)/2m)\,; \\
1-k &\text{with prob }\quad B_{r,s}(F_m(k)/2m)-B_{r,s}(F_m(k-1)/2m)\,; \\
2 &\text{with prob }\quad B_{r,s}(F_m(k-1)/2m)\,.
\end{cases}
\]
\end{lem}

In order for $F_m(k)/2m$ and $F_m(k-1)/2m$ to tend to limits of $x$ and $p$ respectively, we should expect $\mean{F_{m+1}(k)-F_m(k)}\approx 2x$ when $F_m(k)/2m\approx x$ and $F_m(k-1)/2m\approx p$, and so $(k+1)B_{r,s}(p)-kB_{r,s}(x)+1=2x$.

Based on this, write
\[
f_k(x,p)=(k+1)B_{r,s}(p)-kB_{r,s}(x)-2x+1
\]
for $x,p\in[0,1]$ and $k\geq 0$. 

\begin{lem}\label{pk}For any $p\in[0,1)$ there is a unique $x\in(0,1)$ with $f_k(x,p)=0$.
\end{lem}

Define a sequence $(p_k)_{k\geq 0}$ by setting $p_0=0$ and for each $k\geq 0$ letting $p_k$ be the unique value in $(0,1)$ such that $f_k(p_k,p_{k-1})=0$. We show in Theorem \ref{converge} that $F_m(k)/2m$ converges to $p_k$ with probability 1 as $m\to\infty$.

\begin{thm}\label{converge}For any $\varepsilon>0$, $k\geq 0$ and $\alpha<1$,
\[
\prob{|F_{m'}(k)/2m'-p_k|<\varepsilon\quad\forall m'\geq m}=1-O(\exp(m^\alpha))\,.
\]
\end{thm}

We then consider the behaviour of the sequence $p_k$ for $k$ large, and prove the following results. We do not consider the case $r=1$, since this is standard preferential attachment and so has $1-p_k=\Theta(k^{-5})$.

\begin{thm}\label{main}The sequence $p_k$ is increasing with limit $p_*\leq 1$, where $p_*$ is the smallest positive root of $B_{r,s}(p)-2p+1=0$. There exists a function $r(s)$ such that $p_*=1$ if and only if $r<r(s)$, which satisfies $r(s)=2s+o(s)$ but also $r(s)=2s+\omega(\sqrt{s})$.

Provided $s\geq 2$, if $p_*=1$ then $-\log(1-p_k)=\Omega(s^{k})$. The only other case with $r>1$ where $p_*=1$ is $r=2,s=1$, and then $1-p_k=(2+o(1))/\log k$.
\end{thm}


For meek choice, Theorem \ref{main} shows a dichotomy between double exponential decay for $r<r(s)$ and condensation-like behaviour for $r\geq r(s)$.  Calculations show that $r(2)=7$, so the simplest case where the conjecture of \cite{KR2013} fails is $r=7, s=2$.

\section{Convergence of the degree distribution}

In this section we prove that the sequence $(p_k)_{k\geq 0}$ is well-defined, and that $F_m(k)/2m\to p_k$ with probability 1.

\begin{proof}[Proof of Lemma \ref{pk}]
Note that $B_{r,s}(x)$ is continuous and strictly increasing in $x$, and
\begin{align*}
\deriv{}{x}B_{r,s}(x)=&\deriv{}{x}\sum_{i=0}^{s-1}\binom{r}{i}x^{r-i}(1-x)^i \\
=&\sum_{i=0}^{s-1}(r-i)\binom{r}{i}x^{r-i-1}(1-x)^i-\sum_{i=1}^{s-1}i\binom{r}{i}x^{r-i}(1-x)^{i-1} \\
=&\sum_{i=0}^{s-1}r\binom{r-1}{i}x^{r-i-1}(1-x)^i-\sum_{i=1}^{s-1}r\binom{r-1}{i-1}x^{r-i}(1-x)^{i-1} \\
=& r\binom{r-1}{s-1}x^{r-s}(1-x)^{s-1} \,,
\end{align*}
since all the other terms cancel. So
\[
\pderiv{}{x}f_k(x,p)=-kr\binom{r-1}{s-1}x^{r-s}(1-x)^{s-1}-2<0\,,
\]
and, for $p\in [0,1)$,
\begin{align*}
f_k(0,p)&=(k+1)B_{r,s}(p)+1\geq 1 \\
f_k(1,p)&=(k+1)(B_{r,s}(p)-1)< 0\,,
\end{align*}
so provided $p\in[0,1)$ there is a unique $x\in(0,1)$ with $f_k(x,p)=0$.
\end{proof}

To prove convergence, we will use the characterisation of the process given in Lemma \ref{markov}.

\begin{proof}[Proof of Lemma \ref{markov}]For $k\geq 1$, if the new vertex added at time $m+1$ chooses a vertex of degree exceeding $k$ to link to, then $F_{m+1}(k)=F_m(k)+1$, since the sum increases by 1 for the new vertex added, and the old vertex linked to was not included in the sum. If it chooses a vertex of degree less than $k$, then $F_{m+1}(k)=F_m(k)+2$, since the sum increases by one for each of the two vertices. If it chooses a vertex of degree exactly $k$, then that vertex will no longer be included in the sum at time $m+1$, but the new vertex still adds 1, so $F_{m+1}(k)=F_m(k)+1-k$. Since each element of the selected $r$-tuple has degree at most $j$ with probability $F_m(j)/2m$, independently of the others, and the vertex linked to has degree at most $j$ if and only if at least $r-s$ of the elements of the selected $r$-tuple do, the probability that the new vertex links to a vertex of degree at most $j$ is $B_{r,s}(F_m(j)/2m)$, and the result follows.
\end{proof}

\begin{proof}[Proof of Theorem \ref{converge}]
We prove the result by induction on $k$; it is trivial for $k=0$. For $k>0$, we will assume $\varepsilon$ is sufficiently small that $(p_k-\varepsilon,p_k+\varepsilon)\subset(0,1)$.

Write $\beta=\sqrt{\alpha}$. By the induction hypothesis, for any $\delta>0$ the probability that there exists $m'\geqslant m^\beta$ with $|F_{m'}(k-1)/2m'-p_{k-1}|>\delta$ decays exponentially in $(m^\beta)^\beta=m^\alpha$. It is sufficient to prove that the probability that there is some $m''\geqslant m$ with $|F_{m''}(k)/2m''-p_k|>\varepsilon$ but $|F_{m'}(k-1)/2m'-p_{k-1}|<\delta$ for all $m'\geqslant m^\beta$ decays at least exponentially in $m^\alpha$ (for some suitably chosen $\delta$).

Suppose that $F_m(k)/2m<p_k-\varepsilon/3$ and $F_m(k-1)/2m>p_{k-1}-\delta$. Then
\begin{align*}
\mathbb{P}(F_{m+1}(k)-F_m(k)=1)&>1-B_{r,s}(p_k-\varepsilon/3) \\
&>1-B_{r,s}(p_k)
\end{align*}
and
\[
\mathbb{P}(F_{m+1}(k)-F_m(k)=2)>B_{r,s}(p_{k-1}-\delta)\,.
\]
So $F_{m+1}(k)-F_m(k)$ dominates the random variable
\[
X^-=
\begin{cases}
1 &\text{with prob }\quad 1-B_{r,s}(p_k-\varepsilon/3)\,; \\
1-k &\text{with prob }\quad B_{r,s}(p_k-\varepsilon/3)-B_{r,s}(p_{k-1}-\delta)\,; \\
2 &\text{with prob }\quad B_{r,s}(p_{k-1}-\delta)\,.
\end{cases}
\]
Similarly, if $F_m(k)/2m>p_k+\varepsilon/3$ and $F_m(k-1)/2m<p_{k-1}+\delta$ then $F_{m+1}(k)-F_m(k)$ is dominated by the random variable
\[
X^+=
\begin{cases}
1 &\text{with prob }\quad 1-B_{r,s}(p_k+\varepsilon/3)\,; \\
1-k &\text{with prob }\quad B_{r,s}(p_k+\varepsilon/3)-B_{r,s}(p_{k-1}+\delta)\,; \\
2 &\text{with prob }\quad B_{r,s}(p_{k-1}+\delta)\,.
\end{cases}
\]
Note that setting $\varepsilon$ and $\delta$ equal to zero in the above gives a variable with expectation $1+(k+1)B_{r,s}(p_{k-1})-kB_{r,s}(p_k)=2p_k$. Choose $\delta$ sufficiently small that $\mathbb{E}(X^+)<2p_k<\mathbb{E}(X^-)$.

Run the process to time $m^\beta$ and suppose that $F_{m^\beta}(k)/2m^\beta<p_k-\varepsilon/3$. Now the process from time $m^\beta$ until $F_t(k)/2t$ first exceeds $p_k-\varepsilon/3$ is dominated by the process $\tilde{F}_t$ where $\tilde{F}_{m^\beta}=F_{m^\beta}(k)$ and $\tilde{F}_{t+1}-\tilde{F}_t$ are iid $X^-$ variables (provided $|F_t(k-1)/2t-p_{k-1}|<\delta$ in this region). Since $p_k, \tilde{F}_{m^\beta}\in[0,1]$, $\tilde{F}_m$ will exceed $p_k-\varepsilon/3$ provided that the average of $m-m^\beta$ iid $X^-$ variables is at least $2p_k-\varepsilon/3$ for $m$ sufficiently large. The probability that this fails is exponential in $m-m^\beta$ by Hoeffding's inequality. So with suitably high probability $F_t(k)/2t\in(p_k-\varepsilon/3,p_k+\varepsilon/3)$ for some $t\in[m^\beta,m]$. This also holds similarly if $F_{m^\beta}(k)/2m^\beta>p_k+\varepsilon/3$ (and trivially if neither is true).

Now suppose $t>m^\beta$ with $F_t(k)/2t\in(p_k-\varepsilon/3,p_k+\varepsilon/3)$. We will show that with suitably high probability $F_t(k)/2t\in(p_k-\varepsilon,p_k+\varepsilon)$ for all $s>t$. Let $s_1$ be the first time after $t$ with $F_{s_1}(k)/2s_1\notin(p_k-2\varepsilon/3,p_k+2\varepsilon/3)$, and assume wlog that it is below. Let $t_1$ be the next time that $F_{t_1}(k)/2t_1\notin(p_k-\varepsilon,p_k-\varepsilon/3)$. We can again bound $F_s(k)$ from below for $s_1\leqslant s \leqslant t_1$ by a process $\tilde{F}_s$ with $\tilde{F}_{s_1}=F_{s_1}(k)$ and $\tilde{F}_{s+1}-\tilde{F}_s$ iid $X^-$ variables. In order for $\tilde{F}_s/2s$ to reach $p_k-\varepsilon$, the sum of $s-s_1$ iid $X^-$ variables must be at most $(s-s_1)(2p_k-2\varepsilon)-2s_1\varepsilon/3$, that is the average must be at least $2\varepsilon-2s_1\varepsilon/3(s-s_1)$ below its expectation. Since $|X^--\mathbb{E}(X^-)|$ is bounded, this is impossible unless $s-s_1>cs_1$ for some constant $c$, and beyond that point its probability is at most exponential in $s-s_1$ by Hoeffding's inequality. Consequently the sum over all $s$ with $s-s_1>cs_1$ is exponential in $s_1$, and therefore the probability that $F_{t_1}(k)/2t_1<p_k-\varepsilon$ is exponential in $s_1$. Similarly, defining $s_2$ to be the first time after $t_1$ such that $F_{s_2}(k)/2s_2\notin (p_k-2\varepsilon/3,p_k+2\varepsilon/3)$, and $t_2$ to be the first time after $s_2$ we are outside $(p_k-\varepsilon,p_k-\varepsilon/3)\cup(p_k+\varepsilon/3,p_k+\varepsilon)$, the probability that we reach $p_k\pm\varepsilon$ by time $t_2$ is exponential in $s_2$. The probability that we ever reach $p_k\pm\varepsilon$ is therefore at most $\sum_{s>m^\beta}q_s$ where $q_s$ decays exponentially in $s$, so the sum decays exponentially in $m^\beta$.
\end{proof}

\section{The limit of the $p_k$}

In this section we show that $p_k\to p_*$ as $k\to\infty$ and use this to show the dichotomy between double exponential decay and condensation-like behaviour for meek choice.

Write $f(p)$ for $f_k(p,p)$; since
\[
f_k(p,p)=B_{r,s}(p)-2p+1\,,
\]
this is independent of $k$. We have $f(0)=1$ and $f(1)=0$; let $p_*$ be the smallest positive root of $f(p)=0$. Since $f(p)$ is continuous, $f(p)>0$ for $p\in [0,p_*)$.
\begin{lem}
As $k\to\infty$, $p_k$ approaches $p_*$ from below.
\end{lem}
\begin{proof}
First, we claim that $p_k<p_*$ for every $k$. We prove this by induction on $k$; it is true for $k=0$. Note that, for $p\in(0,1)$,
\[
\pderiv{}{p}f_k(x,p)=(k+1)r\binom{r-1}{s-1}p^{r-s}(1-p)^{s-1}>0\,,
\]
so $f_k(x,p)$ is strictly increasing with $p$; we showed in the proof of Lemma \ref{pk} that it is strictly decreasing in $x$. Suppose $p_{k-1}<p_*\leq p_k$. Then
\begin{align*}
0=f_k(p_k,p_{k-1})&\leq f_k(p_*,p_{k-1}) \\
&<f_k(p_*,p_*)=0\,,
\end{align*}
a contradiction. So $p_k<p_*$ for all $k$. It follows that $p_k>p_{k-1}$ for every $k$, since otherwise, as $p_{k-1}<p_*$,
\begin{align*}
f_k(p_k,p_{k-1})\geq f(p_{k-1})>0\,.
\end{align*}

So $p_k\to p_\infty$ for some $p_\infty\in[0,p_*]$. Suppose $p_\infty\neq p_*$; then $p_k\in[0,p_\infty]$ and $f(p)$ attains some minimum $\varepsilon>0$ in this region. So $f(p_k)\geq\varepsilon$ for all $k$. We may bound
\[
\pderiv{}{x}f_k(x,p_{k-1})=-kr\binom{r-1}{s-1}x^{r-s}(1-x)^{s-1}-2
\]
as being greater than $-(ck+2)$, where $c$ is some constant which depends only on $r$ and $s$. Consequently $p_k-p_{k-1}>\varepsilon/(ck+2)$. But then $p_k>\sum_{i=1}^k\varepsilon/(ci+2)$, which is impossible for $k$ sufficiently large.
\end{proof}

We therefore get different behaviour depending on whether $p_*=1$ or $p_*<1$. For $s=1$, $p_*<1$ if $r\geq 3$; for $s=2$, $p_*<1$ if $r\geq 7$; and for $s=3$, $p_*<1$ if $r\geq 10$. Clearly $p_*$ is decreasing in $r$ for fixed $s$, as
\[
f(p)=\prob{\bin{r,1-p}<s}-2p+1\,,
\]
which is decreasing with $r$ for fixed $s$ and $p$.

Next we address the question of how large $r$ needs to be in terms of $s$ to see the degenerate behaviour. Write $r(s)$ for the smallest $r$ for which $p_*<1$; recall that $p_*=1$ if and only if the function $f(p)=B_{r,s}(p)-2p+1$ is strictly positive on $[0,1)$; since trivially $f(p)>0$ for $p\in(0,1/2)$ we shall only bound $f(p)$ for $p\in[1/2,1]$.

\begin{lem}For each $s\geq 1$, $r(s)/s\geq 2$.
\end{lem}
\begin{proof}It is sufficient to prove that $f(p)$ is positive on $[0,1)$ when $r=2s-1$. We claim that $B_{2s-1,s}(p)\geq p$ for $p\geq 1/2$. To prove this claim, note that
\begin{align*}
B_{2s-1,s}(p)&=\sum_{i=s}^{2s-1}\binom{2s-1}{i}p^i(1-p)^{2s-1-i} \\
&=\sum_{i=s}^{2s-1}\binom{2s-1}{i}p^{2s-1-i}(1-p)^{i}\left(\frac{p}{1-p}\right)^{2(i-s)+1} \\
&\geq\frac{p}{1-p}\sum_{i=s}^{2s-1}\binom{2s-1}{i}p^{2s-1-i}(1-p)^{i} \\
&=\frac{p}{1-p}B_{2s-1,s}(1-p)\,,
\end{align*}
so
\[
\frac{B_{2s-1,s}(p)}{1-B_{2s-1,s}(p)}\geq\frac{p}{1-p}\,.
\]
Consequently $f(p)\geq 1-p>0$ for $p\in[1/2,1)$.
\end{proof}

\begin{thm}As $s\to\infty$, $r(s)/s\to 2$.
\end{thm}
\begin{proof}
We show that for any $\eps>0$, $r(s)<(2+2\eps)s$ for $s$ sufficiently large. Suppose $s>1/\eps$, and let $r$ be an integer between $(2+\eps)s$ and $(2+2\eps)s$. Then Hoeffding's inequality gives
\[
B_{r,s}(p)\leq\exp(-2(r(1-p)-s)^2/r)\,.
\]
Now $(r(1-p)-s)>s((2+\eps)(1-p)-1)$, which is positive provided $p<\frac{1+\eps}{2+\eps}$, and in that case $(r(1-p)-s)^2>s^2((2+\eps)(1-p)-1)^2$. Fix $p\in(\frac{1}{2},\frac{1+\eps}{2+\eps})$, then
\[
B_{r,s}(p)<\exp\left(-s((2+\eps)(1-p)-1)^2/(1+\eps)\right)\,,
\]
so there exists $\delta>0$ with $f(p)<\ee^{-\delta s}-2p+1$. Since $-2p+1<0$, it follows that $f(p)<0$ for $s$ sufficiently large.
\end{proof}

\begin{thm}As $s\to\infty$, $s^{-1/2}(r(s)-2s)\to\infty$.
\end{thm}
\begin{proof}
Fix $c>0$ and take $s$ sufficiently large and $r$ such that $2s<r<2s+cs^{1/2}$. Now
\[
f(1/2)=B_{r,s}(1/2)>\tfrac{1}{2}\Phi\Big(\tfrac{-c}{\sqrt{2}}\Big) \,.
\]
Write $\rho=1/2+\Phi(-c/\sqrt{2})/4$; since $f'(p)\geq -2$, $f(p)>0$ for all $p\in[1/2,\rho)$. Also, if $s>c^2$, $r<2s+cs^{1/2}<3s$ and so
\begin{align*}
f'(p)&=r\binom{r-1}{s-1}p^{r-1}(1-p)^{s-1}-2 \\
&<r2^{r-1}(1-p)^{s-1}-2 \\
&<3s2^{3s}(1-p)^{s-1}-2 \\
&<(9-9p)^{s-1}-2<-1 \,,
\end{align*}
if $p>8/9$ and $s$ sufficiently large. Since $f(1)=0$, $f(p)>1-p>0$ for all $p\in(8/9,1)$.

Finally, if $s$ is sufficiently large then $(2s+cs^{1/2})\rho>s+cs^{1/2}+s^{3/4}$, so for any $p\in[\rho,8/9]$,
\[
\prob{\bin{r,p}\leq r-s}\leq\exp(-2s^{3/2}/r)\,,
\]
again by Hoeffding's inequality, and
\[
f(p)\geq 2/9-\exp(-2s^{1/2}/3)\,,
\]
which is positive for $s\geq 6$. So $f(p)>0$ in each of the intervals $[1/2,\rho)$, $[\rho,8/9]$ and $(8/9,1)$, and consequently $r(s)\geq 2s+cs^{1/2}$, for all sufficiently large $s$.
\end{proof}

\section{Meek choice with double exponential decay}

In this section we assume $r\geq s\geq 2$, and show double exponential decay of the limiting distribution in the case where $p_*=1$.  Write $q_k$ for $1-p_k$.
\begin{lem}\label{recurrence}
If $p_k\to 1$, then for $k$ sufficiently large $q_k$ satisfies
\[
q_k<\frac{k+1}{2}\binom{r}{s}q_{k-1}^s\,.
\]
\end{lem}
\begin{proof}
Choose $k$ sufficiently large that $p_{k-1}>(r-s)/(r-1)$. Write $f_k(x)$ for $f_k(x,p_{k-1})$. Recall that
\[
f_k'(x)=-kr\binom{r-1}{s-1}x^{r-s}(1-x)^{s-1}-2\,,
\]
so $f_k'(1)=-2$ and
\[
f_k''(x)=-kr\binom{r-1}{s-1}x^{r-s-1}(1-x)^{s-2}\left[(r-s)(1-x)-(s-1)x\right]\,,
\]
which is positive provided $(1-x)/x<(s-1)/(r-s)$, that is to say provided $x>(r-s)/(r-1)$. So $f_k'(x)<-2$ in this region, and so, provided $\varepsilon<(s-1)/(r-1)$,
\begin{align*}
f_k(1-\varepsilon)&>f_k(1)+2\varepsilon \\
&=(k+1)(B_{r,s}(p_{k-1})-1)+2\varepsilon \,.
\end{align*}
Note that $1-B_{r,s}(p_{k-1})$ is the probability that at least $s$ of $r$ independent events with probability $q_{k-1}=1-p_{k-1}$ occur. Consequently this is bounded by the expected number of $s$-tuples of events which all occur, and that is $\binom{r}{s}q_{k-1}^s$. So, setting
\[
\varepsilon=\frac{k+1}{2}\binom{r}{s}q_{k-1}^s\,,
\]
$f_k(1-\varepsilon)>0$ provided $\varepsilon<(s-1)/(r-1)$. If $\eps>(s-1)/(r-1)$ then $\eps>q_{k-1}>q_k$, so in either case $q_k<\varepsilon$.
\end{proof}

Next we show that Lemma \ref{recurrence} implies doubly-exponential decay provided we can find some $k_0$ with
\begin{equation}
q_{k_0}<\left(\frac{2}{\binom{r}{s}(k_0+3)}\right)^{1/(s-1)}\,.\label{cutoff}
\end{equation}
Lemma \ref{doubexp} is analogous to Lemma 3.3 of \cite{MalPaq1}, but we improve the condition on $y_k$ by bounding $\sum_i s^{-i}\log(k+i)$ more tightly.
\begin{lem}\label{doubexp}Fix integers $s\geq 2$ and $k\geq 0$, and positive reals $a$ and $y_k$. Let $y_i=a(i+1)y_{i-1}^s$ for $i>k$. If $y_k<(a(k+3))^{-1/(s-1)}$ then there is some $C<1$ and $b>0$ such that
\[y_{k+j}<bC^{s^j}\,.\]
\end{lem}
\begin{proof}
Write $y_{k+j}$ in terms of $y_k$ for $j=1,2,\ldots$.
\begin{align*}
y_{k+2}&=a(k+3)y_{k+1}^s \\
&=a^{s+1}y_k^{s^2}(k+2)^s(k+3) \,, \\
y_{k+3}&=a(k+4)y_{k+2}^s \\
&=a^{s^2+s+1}y_k^{s^3}(k+2)^{s^2}(k+3)^s(k+4) \,,
\end{align*}
and so on, so
\[
y_{k+j}=a^{(s^j-1)/(s-1)}y_k^{s^j}(k+2)^{s^{j-1}}(k+3)^{s^{j-2}}\cdots (k+j+1)\,.
\]
So
\begin{align*}
\log\frac{y_{k+j}}{y_k}=&(s^j-1)\log y_k+\frac{s^j-1}{s-1}\log a \\
&+s^{j-1}\log(k+2)+s^{j-2}\log(k+3)+\cdots+\log(k+j+1)\,.
\end{align*}
Now
\[
\frac{\log(k+2)}{s}+\frac{\log(k+3)}{s^2}+\cdots+\frac{\log(k+j+1)}{s^j}<(1-s^{-j})\sum_{i=1}^\infty\frac{\log(k+i+1)}{s^i}
\]
and so
\[
\left(1+\frac{1}{s^j-1}\right)\left(\frac{\log(k+2)}{s}+\cdots+\frac{\log(k+j+1)}{s^j}\right)<\sum_{i=1}^\infty\frac{\log(k+i+1)}{s^i}\,.
\]

We can rewrite the RHS to get, using the fact that $\log$ is increasing and concave (which also implies absolute convergence),
\begin{align*}
&\sum_{i=1}^\infty\frac{\log(k+2)}{s^i}+\sum_{i=2}^\infty\frac{\log(k+3)-\log(k+2)}{s^i}+\sum_{i=3}^\infty\frac{\log(k+4)-\log(k+3)}{s^i}+\cdots \\
&=\frac{\log(k+2)+[\log(k+3)-\log(k+2)]/s+[\log(k+4)-\log(k+3)]/s^2+\cdots}{s-1} \\
&<\frac{\log(k+2)+[\log(k+3)-\log(k+2)](1/s+1/s^2+\cdots)}{s-1} \\
&\leq\frac{\log(k+3)}{s-1}\,.
\end{align*}
So
\[
\frac{1}{s^j-1}\log\frac{y_{k+j}}{y_k}<\log y_k+\frac{\log a+\log(k+3)}{s-1}\,.
\]
If $y_k<(a(k+3))^{-1/(s-1)}$ then the right-hand side is some negative constant $c$, so
\[
y_{k+j}<y_k\mathrm{e}^{cs^j-c}=bC^{s^j}\,,
\]
where $b=y_k\mathrm{e}^c>0$ and $C=\mathrm{e}^c<1$, as required.
\end{proof}

Next we show that a suitable $k_0$ exists provided $p_*=1$; it is sufficient to show $q_k=o(1/k)$ in this case.

\begin{lem}\label{quadratic}
If $q_k\to 0$ then $q_k=O(1/k^2)$.
\end{lem}
\begin{proof}Write $g_k(y,q)=f_k(1-y,1-q)$. Now
\begin{align*}
g_k(y,q)&=(k+1)(B_{r,s}(1-q)-1)-k(B_{r,s}(1-y)-1)+2y \\
&=k\prob{\bin{r,y}\geq s}-(k+1)\prob{\bin{r,q}\geq s}+2y \,.
\end{align*}
Now $\prob{\bin{r,q}\geq s}=\sum_{i=s}^ra_iq^i$, for some integers $a_i$ depending only on $s$ and $r$. Consequently, setting $y=\left(\frac{k-1}{k}\right)^2q_k$,
\[
g_k(y,q_k)=\sum_{i=s}^ra_iq_k^i\left(\frac{(k-1)^{2i}}{k^{2i-1}}-(k+1)\right)+2\left(\frac{k-1}{k}\right)^2q_k\,.
\]
Since $(k-1)^{2i}\geqslant k^{2i}-2ik^{2i-1}$,
\begin{align*}
\frac{(k-1)^{2i}}{k^{2i-1}}-(k+1)&\geqslant \frac{k^{2i}-2ik^{2i-1}-k^{2i}-k^{2i-1}}{k^{2i-1}} \\
&=-(2i+1)\,.
\end{align*}
Furthermore, $\frac{(k-1)^{2i}}{k^{2i-1}}<k+1$, so, provided $k\geqslant 4$,
\begin{align*}
g_k(y,q_k)&\geqslant-\sum_{i=s}^r(2i+1)\max(a_i,0)q_k^i+2\left(\frac{k-1}{k}\right)^2q_k \\
&>q_k-q_k^s\sum_{i=s}^r(2i+1)\max(a_i,0)\,,
\end{align*}
which is positive provided $q_k<\left(\sum_{i=s}^r(2i+1)\max(a_i,0)\right)^{-1/(s-1)}$. So provided $q_k$ is less than some positive constant which depends only on $r$ and $s$, we have $q_{k+1}<\left[\frac{k-1}{k}\right]^2q_k$, and so if $q_k\to 0$ we must have $q_k=O(1/k^2)$. \end{proof}

Combining Lemmas \ref{recurrence}, \ref{doubexp} and \ref{quadratic}, we have shown that for any $s\geq 2$, if $p_*=1$ then $1-p_k$ eventually decays doubly exponentially, with $-\log(1-p_k)=\Omega(s^k)$.

The point at which the doubly-exponential decay starts, that is to say the first value of $k_0$ which satisfies \eqref{cutoff}, varies with $r$ and $s$. For $s=2$ and $r=2,3,4,5$ we can calculate the values of $p_k$ using Maple, which gives
\begin{center}\begin{tabular}{ccc}
$r$ & $k_0$ & $p_{k_0}$ \\
2 & 4 & 0.7761155642 \\
3 & 18 & 0.9793382628 \\
4 & 98 & 0.9977982955 \\
5 & 2416 & 0.9999471884
\end{tabular}\end{center}

For $r=6$ and $s=2$, numerical solutions are impractical, but we can bound the value of $\log k_0$ as being somewhere between 23 and 58. Write $h_k(x)$ for $-\pderiv{}{x}f_k(x,p_{k-1})$; provided $p_k<0.7$ we have $f(p_{k-1})/h_k(0.7)<p_k-p_{k-1}<f(p_{k-1})/h_k(p_{k-1})$, since $h_k$ is increasing in this region. These relations allow us to inductively calculate lower and upper bounds for $p_k$ which give $p_{213778}<0.7<p_{24864713}$.
Since $\prob{\bin{6,q}\geq 2}=15q^2-40q^3+45q^4-24q^5+5q^6$, from Lemma \ref{quadratic} we get that $q_{k+1}<\left[\frac{k-1}{k}\right]^2q_k$ provided $q_k<1/535$; write $k_1$ for the first value of $k$ for which this is satisfied. Splitting the region $(0.7,534/535)$ into regions of size at most $0.01$ and considering the maximum and minimum values of $f(p)$ and $h_k(x)$ in each region allows us to bound the value of the $p_k$ between pairs of harmonic series. These calculations give $23<\log k_1<31$. For $k\geq k_1$, $(k-1)^2q_k$ is decreasing, and it is less than $\mathrm{e}^{62}/535$ at $k_1$. Consequently, if $\log k>58$ then $q_k<\mathrm{e}^4/535(k-1)<2/15(k+3)$, so $\log k_0<58$; certainly $\log k_0>\log k_1>23$.

\section{Greedy choice between two}

Now we turn our attention to the case $r=2,s=1$. Here we have that $p_k$ is the unique solution in $(0,1)$ to
\[
(k+1)p_{k-1}^2-kp_k^2-2p_k+1=0\,;
\]
dividing by $k$ and completing the square gives
\[
\left(p_k+\frac{1}{k}\right)^2=\frac{(k+1)(kp^2_{k-1}+1)}{k^2}\,.
\]

%
We prove the following result.

\begin{thm}\label{21b}For greedy choice with two choices, $(1-p_k)\log(k+1)\to 2$.
\end{thm}

First we show that the limit exists and is at most 2. We show that if it were less than 2 we would have monotonicity, from which we later obtain a contradiction.

\begin{lem}\label{21a}There is some $\alpha\leq 2$ such that $(1-p_k)\log(k+1)\to\alpha$; if $\alpha<2$ then $(1-p_k)\log(k+1)$ is increasing for $k$ sufficiently large.
\end{lem}
\begin{proof}
Suppose that $p_{k-1}>1-\frac{2}{\log k}$ but $p_k\leqslant 1-\frac{2}{\log(k+1)}$. Then we must have
\begin{equation}
\left(1-\frac{2}{\log(k+1)}+\frac{1}{k}\right)^2 > \frac{(k+1)\left(k(1-\tfrac{2}{\log k})^2+1\right)}{k^2}\,. \label{aa}
\end{equation}
The LHS of \eqref{aa} is
\begin{equation}
1+\frac{4}{(\log(k+1))^2}+\frac{1}{k^2}-\frac{4}{\log(k+1)}+\frac{2}{k}-\frac{4}{k\log(k+1)}\,, \label{bb}
\end{equation}
and the RHS is
\begin{equation}
1-\frac{4}{\log k}+\frac{4}{(\log k)^2}+\frac{2}{k}-\frac{4}{k\log k}+\frac{4}{k(\log k)^2}+\frac{1}{k^2}\,. \label{cc}
\end{equation}
Subtracting \eqref{cc} from \eqref{bb} and multiplying by $(\log(k+1)\log k)^2/4$ gives
\[
\left((1+\tfrac{1}{k})\log(k+1)\log k-\log(k+1)-\log k\right)(\log(k+1)-\log k)>\frac{(\log(k+1))^2}{k}\,.
\]
If $(1+\tfrac{1}{k})\log(k+1)\log k-\log(k+1)-\log k<0$ then this is trivially false; otherwise
\begin{align*}
&\left((1+\tfrac{1}{k})\log(k+1)\log k-\log(k+1)-\log k\right)(\log(k+1)-\log k) \\
&<\left((\tfrac{1}{k}+\tfrac{1}{k^2})\log(k+1)\log k-\tfrac{1}{k}\log(k+1)-\tfrac{1}{k}\log k\right) \\
&<\tfrac{1}{k}\log(k+1)\log k+\tfrac{1}{k^2}\log(k+1)\log k-\tfrac{1}{k}\log(k+1) \\
&<\tfrac{1}{k}\log(k+1)\log k<\tfrac{1}{k}(\log(k+1))^2\,,
\end{align*}
a contradiction. (We use the fact that $\frac{1}{k+1}<\log(k+1)-\log k<\frac{1}{k}$, seen by considering $\int_k^{k+1}x^{-1}\mathrm{d}x$.) Consequently if $p_{k-1}>1-\frac{2}{\log k}$ then also $p_k>1-\frac{2}{\log(k+1)}$. Since this holds for $k=2$, by induction it holds for every $k\geqslant 2$.

Conversely, suppose that for some $0<a<2$ and some $k$ we have $p_{k-1}\leqslant 1-\frac{a}{\log k}$ but $p_k\geqslant 1-\frac{a}{\log (k+1)}$.
Then we must have
\[
\left(1-\frac{a}{\log(k+1)}+\frac{1}{k}\right)^2 \leqslant \frac{(k+1)\left(k(1-\tfrac{a}{\log k})^2+1\right)}{k^2}\,,
\]
and the LHS and RHS of this are
\[
1+\frac{a^2}{(\log(k+1))^2}+\frac{1}{k^2}-\frac{2a}{\log(k+1)}+\frac{2}{k}-\frac{2a}{k\log(k+1)}
\]
and
\[
1-\frac{2a}{\log k}+\frac{a^2}{(\log k)^2}+\frac{2}{k}-\frac{2a}{k\log k}+\frac{a^2}{k(\log k)^2}+\frac{1}{k^2}
\]
respectively. We multiply the difference by $k(\log(k+1)\log k)^2/a$ to obtain
\begin{align*}
& 2(k+1)\log(k+1)\log k (\log(k+1)-\log k) \\
&-ak\left((\log(k+1))^2-(\log k)^2\right)-a(\log(k+1))^2\leqslant 0
\end{align*}
Bounding $\log(k+1)-\log k$ as before,
\begin{align*}
& 2(k+1)\log(k+1)\log k (\log(k+1)-\log k) \\
&-ak\left((\log(k+1))^2-(\log k)^2\right)-a(\log(k+1))^2 \\
>& 2\log(k+1)\log k-a(\log(k+1)+\log k)-a(\log(k+1))^2 \\
=&(2-a)\log(k+1)\log k-a(\log(k+1)+\log k) \\
&+a\log(k+1)(\log k-\log(k+1)) \\
>&(2-a)\log(k+1)\log k-a\left(\left(1+\tfrac{1}{k}\right)\log(k+1)+\log k\right) \\
>&\log(k+1)((2-a)\log k-3a)\,,
\end{align*}
which is positive provided $k>\exp(3a/(2-a))$. Consequently, if there exists $j>\exp(3a/(2-a))$ for which $p_j\leqslant1-a/\log(j+1)$ then $p_k<1-a/\log(k+1)$ for every $k>j$. In particular, this holds with $a=1$ and $j=100$.

Let $a_k=(1-p_k)\log(k+1)$ and $\alpha=\lim\sup a_k$. We know that $1<a_k<2$ for $k>100$ and so $1<\alpha\leqslant 2$. We claim that $\lim a_k$ exists, and so $a_k\to \alpha$.

Suppose $\alpha=2$. Then for any $\varepsilon>0$ there exists $j>\exp(6/\varepsilon)$ with $a_j>2-\varepsilon$, and so $a_k>2-\varepsilon$ for all $k>j$. Consequently $a_k\to 2$.

Conversely, suppose $\alpha<2$. Then choose any $\beta$ with $\alpha<\beta<2$. We may find $i$ such that $a_k<\beta$ for all $k>i$. If $j>\max(i,\exp(3\beta/(2-\beta)))$ then $j>\exp(3a_j/(2-a_j))$ so $a_k>a_j$ for all $k>j$. So $a_k$ is strictly increasing for $k$ sufficiently large, and therefore tends to a limit, which must be $\alpha$.
\end{proof}

We complete the analysis by showing that the limit is equal to 2.

\begin{proof}[Proof of Theorem \ref{21b}]
By Lemma \ref{21a}, $(1-p_k)\log(k+1)$ approaches some limit $\alpha\leq 2$. Suppose $\alpha<2$. In this case $a_k$ is eventually increasing so $1<a_k<\alpha$ for all $k$ sufficiently large. Suppose
\[
a_k<a_{k-1}\left(1+\frac{1/2-\alpha/4}{k\log k}\right)\,,
\]
so that
\[
p_k>1-\frac{a_{k-1}}{\log(k+1)}-\frac{a_{k-1}(1/2-\alpha/4)}{k\log k\log(k+1)}\,.
\]
Then we must have
\[
\left(1-\frac{a_{k-1}}{\log(k+1)}-\frac{a_{k-1}(1/2-\alpha/4)}{k\log k\log(k+1)}+\frac{1}{k}\right)^2 < \frac{(k+1)\left(k(1-\tfrac{a_{k-1}}{\log k})^2+1\right)}{k^2}\,.
\]
Again we subtract the LHS from the RHS and multiply by $k(\log(k+1)\log k)^2/a_{k-1}$ to obtain
\begin{align*}
& 2(k+1)\log(k+1)\log k (\log(k+1)-\log k) \\
&-a_{k-1}k\left((\log(k+1))^2-(\log k)^2\right)-a_{k-1}(\log(k+1))^2 \\
&+a_{k-1}(1/2-\alpha/4)^2/k+a_{k-1}(1-\alpha/2)\log k \\
&-(1-\alpha/2)(1+1/k)\log(k+1)\log k\leqslant 0\,.
\end{align*}
If $k$ is sufficiently large we have $a_{k-1}>1>\log(k+1)/k$. Now
\begin{align*}
& 2(k+1)\log(k+1)\log k (\log(k+1)-\log k) \\
&-a_{k-1}k\left((\log(k+1))^2-(\log k)^2\right)-a_{k-1}(\log(k+1))^2 \\
&+a_{k-1}(1/2-\alpha/4)^2/k+a_{k-1}(1-\alpha/2)\log k \\
&-(1-\alpha/2)(1+1/k)\log(k+1)\log k \\
>& 2\log(k+1)\log k-a_{k-1}(\log(k+1)+\log k) \\
&-a_k(\log(k+1))^2-(1-\alpha/2)\log(k+1)\log k \\
=& (1-a_{k-1}+\alpha/2)\log(k+1)\log k \\
&-a_{k-1}\log(k+1)(\log(k+1)-\log k)-a_{k-1}(\log(k+1)+\log k) \\
>& \log(k+1)((1-a_{k-1}+\alpha/2)\log k-3a_{k-1}) \\
>& \log(k+1)((1-\alpha/2)\log k-3\alpha)\,,
\end{align*}
provided $k$ is large enough that $a_{k-1}<\alpha$. If also $k>\exp(6\alpha/(2-\alpha))$ then the final expression above is positive and we have a contradiction. So for some $k_0$ and all $k>k_0$,
\[
a_k\geqslant a_{k-1}\left(1+\frac{1/2-\alpha/4}{k\log k}\right)\,,
\]
so
\[
a_k\geqslant a_{k_0}+a_{k_0}(1/2-\alpha/4)\sum_{j=k_0+1}^k(j\log j)^{-1}\,.
\]
But $\sum_{j=k_0+1}^k(j\log j)^{-1}$ diverges as $k\to\infty$, so this is eventually larger than $\alpha$, a contradiction. So we must have $\alpha=2$.
\end{proof}

\section{Discussion}
\subsection{Comparison with preferential attachment}

The Barab\'asi--Albert preferential attachment tree has limiting degree distribution given by a power law with exponent 3 \cite{BRST}. M\'ori \cite{Mori}, Cooper and Frieze \cite{CF2003}, and Jordan \cite{JJ2006} give generalisations which can produce power laws with any exponent strictly greater than 2. In each of these the preferential choice is made with probability proportional to $w(d(v))$ for some positive linear weight function $w$. Distributions outside this regime may arise if a superlinear or sublinear weight function is chosen. Oliveira and Spencer \cite{OS2005} consider the superlinear case, and show that for any $a>1$ the weight function $w(k)=k^a$ leads to a degenerate distribution, where only finitely many vertices of degree above a certain threshold will ever appear. Rudas, T\'oth and Valk\'o \cite{RTV} give the asymptotic degree distribution for a wide range of sublinear weight functions. Their condition also holds for any bounded weight function, and in the case of constant weights there is exponential decay in the degree distribution; an outline proof of this fact is also given in \cite{BRST}.

In the max choice model, the choice of vertex to attach to is more heavily biased towards vertices of high degree than linear preferential attachment, so we might expect a heavier tail than any such model. For $r=2,s=1$ the proportion of vertices of degree $k$ decays as $(k\log k)^{-2}$; this is a heavier tail than for linear preferential attachment, which gives $k^{-a}$ for some $a>2$, but the $\log k$ factor ensures that it still has a finite mean. For the meek choice model, the probability of attaching to a specific vertex does not necessarily increase with degree, so faster than exponential decay is not surprising.

\subsection{Comparison with Achlioptas processes}

The method of first selecting an $r$-tuple of vertices at random, and then choosing between them according to some rule, resembles Achlioptas processes (see, for example, \cite{RW2012}) in which two or more potential edges are selected at random, and then one of them is chosen to be added to the graph. The principal difference is that Achlioptas processes are a generalisation of the Erd\H{o}s--R\'enyi random graph process, where the number of vertices is fixed and edges are added one by one, whereas these models generalise the Barab{\'a}si--Albert preferential attachment process, where new vertices of degree 1 are added to a tree. In Achlioptas processes, as in the Erd\H{o}s--R\'enyi process, the random selection is uniform, whereas we consider preferential selection of vertices. Replacing the independent preferential choices in our model with independent uniform choices, we obtain a process we might call uniform attachment with choice. D'Souza, Krapivsky and Moore \cite{DKM} consider such processes with several potential choice criteria, including max choice and min choice. We summarise how changing from uniform to preferential attachment affects the degree distributions for the $r=1$ and $r=2$ cases in the following table.

\vspace{2ex}
\begin{tabular}{r|ccc}
\toprule
& max choice & no choice & min choice \\
\midrule
preferential attachment & $(k\log k)^{-2}$ & $k^{-3}$ \cite{BRST} & $\exp(-\Omega(2^k))$ \\
uniform attachment & $(3/2)^{-k}$ \cite{DKM} & $2^{-k}$ \cite{BRST} & $\exp(-\Omega(2^k))$ \cite{DKM} \\
\bottomrule
\end{tabular}

\subsection{Sampling without replacement}

All our analysis in this paper assumes that the preferential choices are made independently, as in \cite{MalPaq1,MalPaq2}. If instead the choice is made without replacement, it is not clear that the values of $F_{m}(k)/2m$ converge, and if they do, the sequence of limits may be different. We can, however, identify a condition under which Theorem \ref{converge} still holds.

Suppose we run the process with choices made without replacement, defining $F_{m}(k)$ as before, and let $P_m$ be the probability that $r$ independent preferential choices made at time $m$ are all different. If $P_m\to 1$ with high probability, then we can prove in the same way that $F_m(k)/2m\to p_k$, for the following reason. Provided $P_m\to 1$ and $F_m(k-1)/2m\to p_{k-1}$, if $m$ is sufficiently large and $F_m(k)/2m<p_k-\varepsilon$, then $F_{m+1}(k)-F_m(k)$ dominates
\[
X^-=
\begin{cases}
1 &\text{with prob }\quad 1-B_{r,s}(p_k-\varepsilon)-\eta\,; \\
1-k &\text{with prob }\quad B_{r,s}(p_k-\varepsilon)-B_{r,s}(p_{k-1}-\delta)-\eta\,; \\
2 &\text{with prob }\quad B_{r,s}(p_{k-1}-\delta)+2\eta\,,
\end{cases}
\]
and we can choose $\delta,\eta$ so as to ensure $\mean{X^-}>2p_k$, as before.

For the process without replacement, then, suppose there exist values $\tilde{p}_k$ such that $F_m(k)/2m\to\tilde{p}_k$ with high probability, and $\tilde{p}_k\to 1$ as $k\to\infty$. In that case, by choice of $k$ the probability of $r$ independent preferential choices including more than one of degree exceeding $k$ can be made arbitrarily small for large $m$. The chance of one of $r-1$ specific vertices of degree at most $k$ being chosen approaches 0 as $m\to\infty$, and so $P_m\to 1$. Consequently $\tilde{p}_k=p_k$ for every $k$. Thus our counterexamples to the conjecture of \cite{KR2013} would still contradict the conjecture even if choice without replacement were assumed.

\section{Acknowledgements}
The first author acknowledges support from the European Union through funding under FP7--ICT--2011--8 project HIERATIC (316705).

\end{document}